\documentclass[11pt,oneside,reqno]{amsart}
\usepackage[all]{xy}
\usepackage{amsfonts,amsmath,amssymb,amscd}
\usepackage[breaklinks]{hyperref}
\usepackage{caption} 
\renewcommand{\(}{\left\(}
\renewcommand{\)}{\right\)}
\renewcommand{\[}{\left\[}
\renewcommand{\]}{\right\]}

\numberwithin{equation}{section}

\newtheorem{theorem}{Theorem}

\newtheorem{lemma}{{\bf Lemma}}

\newtheorem{remark}[subsection]{Remark}

\begin{document}
\title[$K_{\frac{1}{2},w}(x)$ and Humbert functions]{The generalized modified Bessel function  $K_{z,w}(x)$ at $z=1/2$ and Humbert functions}



\author{Rahul Kumar}
\address{Discipline of Mathematics, IIT Gandhinagar \\
Palaj, Gandhinagar\\ Gujarat-382355 \\ India}
\email{rahul.kumr@iitgn.ac.in}



\maketitle
\pagenumbering{arabic}
\pagestyle{headings}

\begin{abstract}
Recently Dixit, Kesarwani, and Moll introduced a generalization $K_{z,w}(x)$ of the modified Bessel function $K_{z}(x)$ and showed that it satisfies an elegant theory similar to $K_{z}(x)$. In this paper, we show that while $K_{\frac{1}{2}}(x)$ is an elementary function, $K_{\frac{1}{2},w}(x)$ can be written in the form of an infinite series of Humbert functions. As an application of this result, we generalize the transformation formula for the logarithm of the Dedekind eta function $\eta(z)$.

\end{abstract}{}

\section{Introduction}

Bessel functions have been studied for almost two hundred years and their theory shows no sign of exhaustion. They have been proved to be extremely useful not only in mathematics but also in Engineering and other sciences. Bessel functions of the first and second kinds of order $z$ are defined by \cite[p. 40]{watson}
\begin{align*}
J_z(x):&=\sum_{n=0}^{\infty}\frac{(-1)^n(x/2)^{2n+z}}{n!\Gamma(n+1+z)}, ~~~~~~~~ |x|<\infty, \nonumber
\end{align*}
and \cite[p. 64]{watson}
\begin{align*}
Y_z(x):=\frac{J_z(x)\cos(\pi z)-J_{-z}(x)}{\sin(\pi z)}
\end{align*}
respectively, where $\Gamma(s)$ is the Gamma function. The modified Bessel functions of the first and second kinds of order $z$ are defined as \cite[p. 77]{watson}
\begin{align*}
I_z(x):=
\begin{array}{cc}
\bigg\{ & 
\begin{array}{cc}
e^{-\frac{1}{2}\pi z i}J_z\left(e^{\frac{1}{2}\pi i}x\right), &\mathrm{if}-\pi<\mathrm{arg}\ x\leq\frac{\pi}{2} \\
e^{-\frac{3}{2}\pi z i}J_z(e^{-\frac{3}{2}\pi i}x), &\mathrm{if}\ \frac{\pi}{2}<\mathrm{arg}\ x\leq\pi,
\end{array}
\end{array}
\end{align*}
and \cite[p. 78]{watson}
\begin{align}\label{kz}
K_z(x):=\frac{\pi}{2}\frac{I_{-z}(x)-I_z(x)}{\sin(z\pi)}
\end{align}
respectively. For the extensive study of the Bessel functions we refer the reader to  \cite{watson}.

Many generalizations of Bessel functions and modified Bessel functions have been studied by several authors. 
Recently, Dixit, Kesarwani, and Moll introduced a different generalization of \eqref{kz} which they termed as the \emph{generalized modified Bessel function $K_{z,w}(x)$}. The motivation behind introducing this function was to  generalize the Ramanujan-Guinand formula \cite[p. 253]{lnb}, which states that for $ab=\pi^2$,
\begin{align}\label{ramguinand}
\sqrt{a}\sum_{n=1}^{\infty}\sigma_{-z}(n)n^{\frac{z}{2}}&K_{\frac{z}{2}}(2na)-\sqrt{b}\sum_{n=1}^{\infty}\sigma_{-z}(n)n^{\frac{z}{2}}K_{\frac{z}{2}}(2nb) \nonumber\\
&=\frac{1}{4}\Gamma\left(\frac{z}{2}\right)\zeta(z)\left\{b^{\frac{1-z}{2}}-a^{\frac{1-z}{2}}\right\}+\frac{1}{4}\Gamma\left(-\frac{z}{2}\right)\zeta(-z)\left\{b^{\frac{1+z}{2}}-a^{\frac{1+z}{2}}\right\},
\end{align}
where $\sigma_{z}(n)$ is the generalized divisor function defined by $\sigma_{z}(n):=\sum_{d|n}d^z$ and $\zeta(s)$ denote the Riemann zeta function. It is known that \cite[Theorem 1.1]{BerndtLeeSohn} this formula is equivalent to the functional equation of the non-holomorphic Eisenstein series on SL$_2(\mathbb{Z})$.

Dixit, Kesarwani, and Moll defined the generalized modified Bessel function  $K_{z,w}(x)$ for $z,w\in\mathbb{C}, x\in\mathbb{C}\backslash\{x\in\mathbb{R}:x\leq0\}$, and Re$(s)>\pm$Re$(z)$, by  \cite[Equation (1.3)]{DixAas}: 
\begin{align}\label{defKzw}
K_{z,w}(x):&=\frac{1}{2\pi i}\int_{(c)}\Gamma\left(\frac{s-z}{2}\right)\Gamma\left(\frac{s+z}{2}\right){}_1F_1\left(\frac{s-z}{2};\frac{1}{2};\frac{-w^2}{4}\right){}_1F_1\left(\frac{s+z}{2};\frac{1}{2};\frac{-w^2}{4}\right)\nonumber\\
&\qquad\qquad\times 2^{s-2}x^{-s}\mathrm{d}s,
\end{align}
where $c:=\mathrm{Re}(s)>\pm \mathrm{Re}(z)$ and ${}_1F_1(a;c;z)$ is the confluent hypergeometric function defined by \cite[p. 172, Equation (7.3)]{temme}
\begin{align}\label{1f1}
{}_1F_1(a;c;z):=\sum_{m=0}^{\infty}\frac{(a)_m}{(c)_m}\frac{z^m}{m!},
\end{align}
with $(a)_m:=a(a+1)(a+2)...(a+m-1)$ for $a\in\mathbb{C}.$ Here, and throughout the paper, $\int_{(c)}$ denotes the line integral $\int_{c-\infty}^{c+i\infty}$.\\
 It can be seen that by letting $w=0$ in \eqref{defKzw}, we get \cite[p. 115, Formula 11.1]{oberhettinger}
\begin{align*}
K_{z,0}(x):&=\frac{1}{2\pi i}\int_{(c)}\Gamma\left(\frac{s-z}{2}\right)\Gamma\left(\frac{s+z}{2}\right)2^{s-2}x^{-s}\mathrm{d}s\nonumber\\
&=K_z(x),
\end{align*}
where, $K_z(x)$ is the usual modified Bessel function \eqref{kz}. The  authors of \cite{DixAas} initiated the theory of $K_{z,w}(x)$ and showed that it is as rich as that of $K_z(x)$. They also gave a beautiful generalization of the aforementioned Ramanujan-Guinand formula \eqref{ramguinand}. We state their remarkable result \cite[Theorem 1.4]{DixAas} in the following theorem.

\begin{theorem}\label{DixRamGuin}
Let $z,w\in\mathbb{C}$. Let $K_{z,w}(x)$ be defined in \eqref{defKzw}. For $a,b>0$ such that $ab=\pi^2$,
{\allowdisplaybreaks\begin{align}\label{genRamGuin}
&\sqrt{a}\sum_{n=1}^\infty \sigma_{-z}(n)n^{z/2}e^{-\frac{w^2}{4}}K_{\frac{z}{2},iw}(2na)-\sqrt{b}\sum_{n=1}^\infty \sigma_{-z}(n)n^{z/2}e^{\frac{w^2}{4}}K_{\frac{z}{2},w}(2nb)\nonumber\\
&=\frac{1}{4}\Gamma\left(\frac{z}{2}\right)\zeta(z)\left\{b^{\frac{1-z}{2}}{}_1F_1\left(\frac{1-z}{2};\frac{1}{2};\frac{w^2}{4}\right)-a^{\frac{1-z}{2}} {}_1F_1\left(\frac{1-z}{2};\frac{1}{2};-\frac{w^2}{4}\right)\right\} \nonumber\\
&\qquad +\frac{1}{4}\Gamma\left(-\frac{z}{2}\right)\zeta(-z)\left\{b^{\frac{1+z}{2}}{}_1F_1\left(\frac{1+z}{2};\frac{1}{2};\frac{w^2}{4}\right)-a^{\frac{1+z}{2}} {}_1F_1\left(\frac{1+z}{2};\frac{1}{2};-\frac{w^2}{4}\right)\right\}.
\end{align}}
\end{theorem}

The modest goal of this work to add the further little piece to this theory.

It is well known that $K_z(x)$ reduces to an elementary function when $z=\frac{1}{2}$, namely, \cite[p. 978, Formula 8.469, no. 3]{grad} 
\begin{align}\label{khalf}
K_{\frac{1}{2}}(x)=\sqrt{\frac{\pi}{2x}}e^{-x}.
\end{align} 
In this work we obtain a new representation for $K_{\frac{1}{2},w}(x)$ which reduces to \eqref{khalf} when $w=0$. This new representation for $K_{\frac{1}{2},w}(x)$ is in terms of a series of Humbert functions or confluent hypergeometric functions of two variables. The study of Humbert functions was initiated by P. Humbert \cite{humbert} in 1922. He defined these functions as a limiting case of Appell functions \cite{appell}, which is why some authors call them `confluent Appell functions'. For more information on Humbert functions, we refer the reader to \cite{brychkov}, \cite{brychkovKimRathie}, \cite{humbert}, \cite{srivastava} and the references therein. Humbert functions are defined by \cite[pp. 58-59, Equations (38-40)]{srivastava}
\begin{align}
\Phi_1(a,b;c;x,y)&=\sum_{m,n=0}^{\infty}\frac{(a)_{m+n}(b)_m}{(c)_{m+n}}\frac{x^my^n}{m!\ n!},\ |x|<1,\ |y|<\infty,\nonumber\\
\Phi_2(a,a';c;x,y)&=\sum_{m,n=0}^{\infty}\frac{(a)_m(a')_n}{(c)_{m+n}}\frac{x^my^n}{m!\ n!},\ |x|<\infty,\ |y|<\infty,\nonumber\\
\Phi_3(a;c;x,y)&=\sum_{m,n=0}^{\infty}\frac{(a)_m}{(c)_{m+n}}\frac{x^my^n}{m!\ n!},\ |x|<\infty,\ |y|<\infty\label{humbert3}.
\end{align}

We state our main result below.
\begin{theorem}\label{K_(half,w)}
Let $\Phi_3(a;c;x,y)$ be defined by \eqref{humbert3}. Let $w\in\mathbb{C},\ x\in\mathbb{C}\backslash\{x\in\mathbb{R}:x\leq0\}$. Then
\begin{align*}
K_{\frac{1}{2},w}(x)=\sqrt{\frac{\pi}{2x}}e^{-x}\sum_{r=0}^{\infty}\frac{\left(\frac{w^2x^2}{64}\right)^r}{r!\left(\frac{1}{2}\right)_r\left(\frac{1}{2}\right)_{2r}}\Phi_3\left(\frac{1}{2};\frac{1}{2}+2r;-\frac{w^2}{4},-\frac{w^2x}{4}\right).
\end{align*}
\end{theorem}
\noindent
Some asymptotic results for $\Phi_3(a;c;x,y)$ have been obtained by Wald and Henkel in \cite{waldHenkel}. 

\begin{remark}
	It can be clearly seen that the expression in \eqref{khalf} for $K_{\frac{1}{2}}(x)$ is just the first term of the above series.
\end{remark}

\begin{remark}
From \cite[Theorem 1.12]{DixAas}, for large values of $|x|$ and $|\arg(x)|<\frac{\pi}{4}$ and $w,z\in\mathbb{C}$, we have
{\allowdisplaybreaks\begin{align}\label{asymptoticlarge}
K_{z,w}(x)=\frac{1}{2}\sqrt{\frac{\pi}{2x}}e^{-x}\left(\cos(w\sqrt{2x})P-\sin(w\sqrt{2x})Q+e^{-\frac{1}{4}w^2}R\right),
\end{align}	}
where
{\allowdisplaybreaks\begin{align*}
P&=1+\frac{32z^2-3w^2-8}{64x}+O\left(x^{-2}\right),\\
Q&=\frac{w}{4\sqrt{2x}}+O\left(x^{-\frac{3}{2}}\right),\\
R&=1+\frac{(4z^2-1)(2-w^2)}{16x}+O\left(x^{-2}\right).
\end{align*}}
Therefore, by using Theorem \ref{K_(half,w)} and \eqref{asymptoticlarge}, it is easy to see that for large value of $|x|$ and $|\arg x|<\frac{\pi}{4}$, we have
\begin{align*}
\sum_{r=0}^{\infty}\frac{\left(\frac{w^2x^2}{64}\right)^r}{r!\left(\frac{1}{2}\right)_r\left(\frac{1}{2}\right)_{2r}}\Phi_3\left(\frac{1}{2};\frac{1}{2}+2r;-\frac{w^2}{4},-\frac{w^2x}{4}\right)&=\frac{1}{2}\Big(\cos(w\sqrt{2x})P-\sin(w\sqrt{2x})Q\nonumber\\
&\qquad\qquad +e^{-\frac{1}{4}w^2}R\Big),
\end{align*}
For small values of $x$, $|\arg(x)|<\frac{\pi}{4}$, one can use \cite[Theorem 1.13 (i)]{DixAas} to get the asymptotics for the series on left-hand side of the above equation.
\end{remark}

The transformation formula for the logarithm of the Dedekind eta function \cite[p.~253]{lnb} was proved by Berndt, Lee and Sohn in \cite[p.~28, Entry 3.3]{BerndtLeeSohn} by using \eqref{ramguinand}.
\begin{theorem}\label{ramFormula}
Let $a$ and $b$ be positive numbers such that $ab=\pi^2$. Then
\begin{equation}\label{ramformula}
\sum_{n=1}^\infty \sigma_{-1}(n)e^{-2na}-\sum_{n=1}^\infty \sigma_{-1}(n)e^{-2nb}=\frac{b-a}{12}+\frac{1}{4}\log\frac{a}{b}.
\end{equation}
\end{theorem}

As an application of our Theorem \ref{K_(half,w)}, we obtained the following  generalization of \eqref{ramformula}.
\begin{theorem}\label{genOfRamanujanFormula}
Let $\Phi_3(a;c;x,y)$ be defined in \eqref{humbert3}. 
Let $a$ and $b$ be positive numbers such that $ab=\pi^2$. Then
\begin{align}\label{eqngenOfRamanujanFormula}
&e^{-\frac{w^2}{4}}\sum_{n=1}^\infty \sum_{r=0}^{\infty}\frac{\sigma_{-1}(n)e^{-2na}}{r!\left(\frac{1}{2}\right)_r\left(\frac{1}{2}\right)_{2r}}\Phi_3\left(\frac{1}{2};\frac{1}{2}+2r;\frac{w^2}{4},\frac{w^2n a}{2}\right)\left(-\frac{wna}{4}\right)^{2r}\nonumber\\
&\qquad- e^{\frac{w^2}{4}}\sum_{n=1}^\infty \sum_{r=0}^{\infty}\frac{\sigma_{-1}(n)e^{-2nb}}{r!\left(\frac{1}{2}\right)_r\left(\frac{1}{2}\right)_{2r}}\Phi_3\left(\frac{1}{2};\frac{1}{2}+2r;-\frac{w^2}{4},-\frac{w^2n b}{2}\right)\left(\frac{wnb}{4}\right)^{2r}\nonumber\\
&=\frac{1}{4}\left[\log\frac{a}{b}-\frac{w^2}{2}\left({}_2F_2\left(1,1;\frac{3}{2},2;\frac{w^2}{4}\right)+{}_2F_2\left(1,1;\frac{3}{2},2;-\frac{w^2}{4}\right)\right)\right] \nonumber \\
&\qquad\qquad +\frac{1}{12}\left[b\left(1+\frac{w\sqrt{\pi}}{2}e^{\frac{w^2}{4}}\mathrm{erf}\left(\frac{w}{2}\right)\right)-a\left(1-\frac{w\sqrt{\pi}}{2}e^{-\frac{w^2}{4}}\mathrm{erfi}\left(\frac{w}{2}\right)\right)\right],
\end{align}
\end{theorem}
\noindent
 where, $\mathrm{erf}(z)$ and $\mathrm{erfi}(z)$ are \emph{error function} and \emph{imaginary error function} respectively, defined by \cite[p. 275]{temme}
 \begin{align}\label{erf}
 \mathrm{erf}(w)=\frac{2}{\sqrt{\pi}}\int_0^w e^{-t^2}\mathrm{d}t,
 \end{align}
 and \cite[p. 32]{janke}\footnote{Factor $2/\sqrt{\pi}$ is missed from the definition in \cite{janke}. }
 \begin{align}\label{erfi}
 \mathrm{erfi}(w)=\frac{2}{\sqrt{\pi}}\int_0^w e^{t^2}\mathrm{d}t.
 \end{align}

\begin{remark}
	By letting $w=0$ and noting that
	\begin{align*}
	\textrm{erf}(0)=0=\textrm{erfi}(0),\
	{}_2F_2(a,b;c,d,0)=1,\
	\Phi_3(b;c;0,0)=1
	\end{align*}
	in Theorem \ref{genOfRamanujanFormula} we get  \eqref{ramformula}. 
\end{remark}

\section{Preliminaries}

Stirling's formula for $\Gamma(s), s=\sigma+it$, in a vertical strip $a \leq\sigma\leq b$ is given by 
\begin{align}\label{stirling}
|\Gamma(s)|=(2\pi)^{\frac{1}{2}}|t|^{\sigma-\frac{1}{2}}e^{-\frac{1}{2}\pi|t|}\left(1+O\left(\frac{1}{|t|}\right)\right)
\end{align}
as $|t|\rightarrow\infty$. The Duplication formula and reflection formula  for the Gamma function are respectively given by \cite[p. 46, Equation (3.4), (3.5)]{temme}
{\allowdisplaybreaks\begin{align}
\Gamma(s)\Gamma\left(s+\frac{1}{2}\right)=\frac{\sqrt{\pi}}{2^{2s-1}}\Gamma(2s). \label{duplication}\\
\Gamma(s)\Gamma(1-s)=\frac{\pi}{\sin(\pi s)},\ s\notin\mathbb{Z}, \label{reflection}
\end{align} }%
Kummer's first transformation \cite[p.173, Equation 7.5]{temme} for the confluent hypergeometric function is given by 
\begin{align}\label{kummer}
{}_1F_1(a;c;z)=e^z{}_1F_1(c-a;c;-z).
\end{align}


We now prove a lemma, which will be used to prove Theorem \ref{K_(half,w)}. This result can be obtained from \cite[p. 332, \textbf{Formula (13.10.10)}]{nist}, however, we derive it here so as to make the paper self-contained.
\begin{lemma}\label{InverseMellinTransform}
	Let $x\in\mathbb{C}\backslash\{x\in\mathbb{R},\ x\leq 0\}$ and $n\in\mathbb{N}\cup \{0\}$. Then For $\frac{1}{2}<c:=\mathrm{Re}(s)<1$, we have
	\begin{align}\label{InverseMellinTransformResult}
	\frac{1}{2\pi i} \int_{(c)}\frac{\Gamma\left(s-\frac{1}{2}\right)\Gamma(s+n)}{\Gamma(s)}x^{-s}ds=(-1)^n\sqrt{\frac{\pi}{x}}\frac{{}_1F_1\left(\frac{1}{2};\frac{1}{2}-n;-x\right)}{\Gamma\left(\frac{1}{2}-n\right)}.
	\end{align}
	Thus, for real $x$,
	\begin{align}\label{eqv}
	\int_{0}^{\infty} {}_1F_1\left(\frac{1}{2};\frac{1}{2}-n;-x\right)x^{s-\frac{3}{2}}dx=\frac{(-1)^n}{\sqrt{\pi}}\frac{\Gamma\left(\frac{1}{2}-n\right)\Gamma\left(s-\frac{1}{2}\right)\Gamma(s+n)}{\Gamma(s)}.
	\end{align}
\end{lemma}
	\begin{proof} The integral on the right-hand side of \eqref{eqv} is convergent for $\frac{1}{2}<\mathrm{Re}(s)<1$. This can be seen by the following argument. From \cite[p. 189, Exercise 7.7]{temme}, we have
		{\allowdisplaybreaks	\begin{align}\label{asymptoteof1F1}
			\frac{1}{\Gamma(c)}{}_1F_1(a;c;z)\sim& \frac{e^zz^{a-c}}{\Gamma(a)}\sum_{n=0}^{\infty}\frac{(c-a)_n(1-a)_n}{n!}z^{-n}\nonumber\\
			&\quad+\frac{e^{\pm\pi ia}z^{-a}}{\Gamma(c-a)}\sum_{n=0}^{\infty}\frac{(a)_n(1+a-c)_n}{n!}(-z)^{-n},
			\end{align}}%
		where the upper sign is taken if $-\frac{1}{2}\pi<\arg (z)<\frac{3}{2}\pi$ and the lower sign if $-\frac{3}{2}\pi<\arg (z)<\frac{1}{2}\pi$. The first part is dominant when Re$(z)>0$; the second part becomes dominant when Re$(z)<0$. In our case we have $z=x>0$ therefore second part of \eqref{asymptoteof1F1} is dominant as $-x<0$, i.e., 
		\begin{align*}
		\frac{1}{\Gamma(c)}{}_1F_1(a;c;-x)\sim \frac{e^{\pi ia}(-1)^ax^{-a}}{\Gamma(c-a)},\ \mathrm{as}\ x\rightarrow \infty,
		\end{align*} 
		and we know that ${}_1F_1(a;c;-x)\sim 1 ,\ \mathrm{as}\ x\rightarrow 0 $. Hence by these two results it is clear that integral $\displaystyle\int_{0}^{\infty}\frac{ {}_1F_1\left(\frac{1}{2};\frac{1}{2}-n;-x\right)}{\sqrt{x}}x^{s-1}ds$ is convergent for $\frac{1}{2}<$Re$(s)<1$.
		
	    We will now prove \eqref{InverseMellinTransformResult}. First assume $0<|x|<1$. Consider the contour formed by the line segments $[-M-\frac{1}{2}-iT,c-iT],\ [c-iT,c+iT],\ \left[c+iT,-M-\frac{1}{2}+iT\right]$ and \newline$\left[-M-\frac{1}{2}+iT,-M-\frac{1}{2}-iT \right]$, where $M$ is a positive integer. It can be easily seen that integrand $\displaystyle\frac{\Gamma\left(s-\frac{1}{2}\right)\Gamma(s+n)}{\Gamma(s)}x^{-s}$ has poles at $s=\frac{1}{2}-m,\ m\in\mathbb{Z},\ 0\leq m<M$, inside the contour because of the poles of $\Gamma\left(s-\frac{1}{2}\right)$. Note that the poles of $\Gamma(s+n)$ at $s=-n,\ n\in\mathbb{Z},\ n\ge 0$ are cancelled by the poles of $\Gamma(s)$. Therefore by the Cauchy residue theorem,
		\begin{align}\label{cauchy1}
		&\left[\int_{-M-\frac{1}{2}-iT}^{c-iT}+\int_{c-iT}^{c+iT}+\int_{c+iT}^{-M-\frac{1}{2}+iT}+\int_{-M-\frac{1}{2}+iT}^{-M-\frac{1}{2}-iT}\right]\frac{\Gamma\left(s-\frac{1}{2}\right)\Gamma(s+n)}{\Gamma(s)}x^{-s} ds\nonumber\\
		&\hspace{1cm}=2\pi i \sum_{m=0}^{M}R_{\frac{1}{2}-m},
		\end{align}
		where $R_{\frac{1}{2}-m}$ is the residue of the integrand at $s=\frac{1}{2}-m$. Now
		\begin{align}\label{residue1}
		R_{\frac{1}{2}-m}&=\lim\limits_{s\rightarrow \frac{1}{2}-m}\left(s-\left(\frac{1}{2}-m\right)\right)\frac{\Gamma\left(s-\frac{1}{2}\right)\Gamma(s+n)}{\Gamma(s)}x^{-s}\nonumber\\
		&=\frac{(-1)^m}{m!}\frac{\Gamma\left(n-m+\frac{1}{2}\right)}{\Gamma\left(\frac{1}{2}-m\right)}x^{-\left(\frac{1}{2}-m\right)}.
		\end{align}
		By the application of Stirling's formula for Gamma function \eqref{stirling}, we can see that the integrals along horizontal segments go to zero as $T\rightarrow \infty$. So by \eqref{cauchy1} and \eqref{residue1}, we have
		\begin{align}\label{imt1}
		\left[\int_{c-i\infty}^{c+i\infty}+\int_{-M-\frac{1}{2}+i\infty}^{-M-\frac{1}{2}-i\infty}\right]\frac{\Gamma\left(s-\frac{1}{2}\right)\Gamma(s+n)}{\Gamma(s)}x^{-s} ds=2\pi i \sum_{m=0}^{M}\frac{(-1)^{m}x^{m-\frac{1}{2}}}{m!}\frac{\Gamma\left(n-m+\frac{1}{2}\right)}{\Gamma\left(\frac{1}{2}-m\right)}.
		\end{align}
		Next,
{\allowdisplaybreaks	\begin{align}
		&\left|\int_{-M-\frac{1}{2}+i\infty}^{-M-\frac{1}{2}-i\infty}\frac{\Gamma\left(s-\frac{1}{2}\right)\Gamma(s+n)}{\Gamma(s)}x^{-s} ds\right| \nonumber\\
		&\qquad=\left|i\int_{-\infty}^{\infty}\frac{\Gamma\left(-M-1+it\right)\Gamma\left(-M+n-\frac{1}{2}+it\right)}{\Gamma\left(-M-\frac{1}{2}+it\right)} x^{M+\frac{1}{2}-it}dt\right| \nonumber\\
		&\qquad =|x|^{M+\frac{1}{2}}\int_{-1}^{1}O(1)dt+|x|^{M+\frac{1}{2}}\int_{1}^{\pm\infty}O\left(|t|^{-M-\frac{3}{2}+n}e^{-\frac{\pi}{2}|t|}\right)dt \nonumber\\
		&\qquad=O\left(|x|^{M+\frac{1}{2}}\right)\nonumber.
		\end{align}}
		Since $|x|<1$, this implies that 
		\begin{align}\label{imt2}
		\lim\limits_{M\rightarrow\infty}\int_{-M-\frac{1}{2}+i\infty}^{-M-\frac{1}{2}+i\infty}\frac{\Gamma\left(s-\frac{1}{2}\right)\Gamma(s+n)}{\Gamma(s)}x^{-s} ds=0,
		\end{align}
		From \eqref{imt1} and \eqref{imt2}, we have
{\allowdisplaybreaks\begin{align}
		\frac{1}{2\pi i}&\int_{c-i\infty}^{c+i\infty}\frac{\Gamma\left(s-\frac{1}{2}\right)\Gamma(s+n)}{\Gamma(s)}x^{-s} ds\nonumber\\
		&=\sum_{m=0}^{\infty}\frac{(-1)^{m}x^{m-\frac{1}{2}}}{m!}\frac{\Gamma\left(n-m+\frac{1}{2}\right)}{\Gamma\left(\frac{1}{2}-m\right)}. \nonumber\\
		&=\sum_{m=0}^{\infty}\frac{(-1)^{m}x^{m-\frac{1}{2}}}{m!}\frac{\Gamma\left(\frac{1}{2}+m\right)\cos\left(\pi m\right)}{\Gamma\left(\frac{1}{2}+m-n\right)\cos(\pi(m-n))} \nonumber\\
		&=\sum_{m=0}^{\infty}\frac{(-1)^{m}x^{m-\frac{1}{2}}}{m!}\frac{\Gamma\left(\frac{1}{2}+m\right)(-1)^m}{\Gamma\left(\frac{1}{2}+m-n\right)(-1)^{m-n}} \nonumber\\
		&={\sqrt{\frac{\pi} {x}}} \sum_{m=0}^{\infty}\frac{(-1)^{m}x^{m}}{m!}\frac{\Gamma\left(\frac{1}{2}+m\right)}{\Gamma\left(\frac{1}{2}\right)}\frac{(-1)^n\Gamma\left(\frac{1}{2}-n\right)}{\Gamma\left(\frac{1}{2}-n\right)\Gamma\left(\frac{1}{2}-n+m\right)} \nonumber\\
		&=(-1)^n{\sqrt{\frac{\pi} {x}}}\frac{{}_1F_1\left(\frac{1}{2};\frac{1}{2}-n;-x\right)}{\Gamma\left(\frac{1}{2}-n\right)},\nonumber
		\end{align}}%
		where, in the second step, we used the reflection formula \eqref{reflection}. As both sides of \eqref{InverseMellinTransformResult} are analytic as functions of $x$, by the principle of the analytic continuation, our result is true for all $x\in\mathbb{C}\backslash\{x\in\mathbb{R},\ x\leq 0\}.$ 
			\end{proof}

\section{Proof of Theorem \ref{K_(half,w)}}

\begin{proof}
The series in the statement of the Theorem is convergent as $\Phi_3\left(\frac{1}{2};\frac{1}{2}+2r;-\frac{w^2}{4},-\frac{w^2}{4}\right)$ tends to 1 as $r\rightarrow \infty$. Let $z=\frac{1}{2}$ in \eqref{defKzw} then for $c:=$Re$(s)>1/2$, we have
\begin{align*}
K_{\frac{1}{2},w}(x)&=\frac{1}{2\pi i}\int_{(c)}\Gamma\left(\frac{s}{2}-\frac{1}{4}\right)\Gamma\left(\frac{s}{2}+\frac{1}{4}\right){}_1F_1\left(\frac{s}{2}-\frac{1}{4};\frac{1}{2};\frac{-w^2}{4}\right){}_1F_1\left(\frac{s}{2}+\frac{1}{4};\frac{1}{2};\frac{-w^2}{4}\right)\nonumber\\
&\hspace{7cm}\times 2^{s-2}x^{-s}\mathrm{d}s.
\end{align*}
Now use \eqref{duplication} in above equation to arrive at
\begin{align}\label{khalf1}
K_{\frac{1}{2},w}(x)&=\frac{1}{2\pi i}\int_{(c)}\sqrt{\frac{\pi}{2}}\Gamma\left(s-\frac{1}{2}\right){}_1F_1\left(\frac{s}{2}-\frac{1}{4};\frac{1}{2};\frac{-w^2}{4}\right){}_1F_1\left(\frac{s}{2}+\frac{1}{4};\frac{1}{2};\frac{-w^2}{4}\right)x^{-s}\mathrm{d}s.
\end{align}
From \cite[Equation (72)]{chaundy},
\begin{align}\label{chaundy1}
{}_1F_1(a;c;x){}_1F_1(a';c;x)=\sum_{r=0}^{\infty}\frac{(a)_r(a')_r}{r!(c)_r(c)_{2r}}x^{2r}{}_1F_1(a+a'+2r;c+2r;x).
\end{align}
Use \eqref{khalf1} and \eqref{chaundy1} to see that
{\allowdisplaybreaks\begin{align}\label{I}
K_{\frac{1}{2},w}(x)&=\sqrt{\frac{\pi}{2}}\frac{1}{2\pi i}\int_{(c)}\Gamma\left(s-\frac{1}{2}\right)\sum_{r=0}^{\infty}\frac{\left(\frac{s}{2}-\frac{1}{4}\right)_r\left(\frac{s}{2}+\frac{1}{4}\right)_r}{r!\left(\frac{1}{2}\right)_r \left(\frac{1}{2}\right)_{2r}}{}_1F_1\left(s+2r;\frac{1}{2}+2r;-\frac{w^2}{4}\right)\nonumber\\
&\qquad\qquad\qquad\times \left(w^4/16\right)^r x^{-s} \mathrm{d}s \nonumber\\
&=\sqrt{\frac{\pi}{2}}\frac{1}{2\pi i}\int_{(c)}\Gamma\left(s-\frac{1}{2}\right){}_1F_1\left(s;\frac{1}{2};\frac{-w^2}{4}\right) x^{-s} \mathrm{d}s \nonumber\\
&\quad+\sqrt{\frac{\pi}{2}}\frac{1}{2\pi i}\int_{(c)}\Gamma\left(s-\frac{1}{2}\right)\sum_{r=1}^{\infty}\frac{\left(\frac{s}{2}-\frac{1}{4}\right)_r\left(\frac{s}{2}+\frac{1}{4}\right)_r}{r!\left(\frac{1}{2}\right)_r \left(\frac{1}{2}\right)_{2r}}{}_1F_1\left(s+2r;\frac{1}{2}+2r;\frac{-w^2}{4}\right)\nonumber\\
&\qquad\qquad\qquad\quad\times\left(w^4/16\right)^r x^{-s} \mathrm{d}s \nonumber\\
&=\sqrt{\frac{\pi}{2}}(I_1+I_2),
\end{align}}%
where 
\begin{align}\label{I1}
I_1:=\frac{1}{2\pi i}\int_{(c)}\Gamma\left(s-\frac{1}{2}\right){}_1F_1\left(s;\frac{1}{2};\frac{-w^2}{4}\right) x^{-s} \mathrm{d}s,
\end{align}
and 
{\allowdisplaybreaks\begin{align}\label{I2}
I_2:=\frac{1}{2\pi i}\int_{(c)}\Gamma\left(s-\frac{1}{2}\right)\sum_{r=1}^{\infty}\frac{\left(\frac{s}{2}-\frac{1}{4}\right)_r\left(\frac{s}{2}+\frac{1}{4}\right)_r}{r!\left(\frac{1}{2}\right)_r \left(\frac{1}{2}\right)_{2r}}{}_1F_1\left(s+2r;\frac{1}{2}+2r;\frac{-w^2}{4}\right)\left(\frac{w^4}{16}\right)^r x^{-s} \mathrm{d}s.
\end{align}}%
We first evaluate $I_1$. To do so, we use the series definition of confluent hypergeometric function in \eqref{I1} and interchange the order of summation and integration because of absolute convergence to deduce that
{\allowdisplaybreaks\begin{align}\label{I11}
I_1&=\sum_{n=0}^{\infty}\frac{(-w^2/4)^n}{(1/2)_n n!} \frac{1}{2\pi i}\int_{(c)}\frac{\Gamma\left(s-\frac{1}{2}\right)\Gamma(s+n)}{\Gamma(s)}x^{-s} ds \nonumber\\
&=\sum_{n=0}^{\infty}\frac{(-w^2/4)^n}{(1/2)_n n!}\left((-1)^n\sqrt{\frac{\pi}{x}}\frac{{}_1F_1\left(\frac{1}{2};\frac{1}{2}-n;-x\right)}{\Gamma\left(\frac{1}{2}-n\right)}\right) \nonumber\\
&=\sqrt{\frac{\pi}{x}}\sum_{n=0}^{\infty} \frac{(-w^2/4)^n (-1)^n}{ n! (1/2)_n \Gamma\left(\frac{1}{2}-n\right)} {}_1F_1\left(\frac{1}{2};\frac{1}{2}-n;-x\right) \nonumber\\
&=\sqrt{\frac{\pi}{x}}e^{-x}\sum_{n=0}^{\infty} \frac{(w^2/4)^n}{ n! (1/2)_n \Gamma\left(\frac{1}{2}-n\right)} {}_1F_1\left(-n;\frac{1}{2}-n;x\right),
\end{align}}%
where in the second step we employed Lemma \ref{InverseMellinTransform} and \eqref{kummer} in the ultimate step .\\
From \cite[p. 152]{temme},
\begin{align}\label{temme1}
L_n^\alpha(x)=\frac{\Gamma(n+\alpha+1)}{\Gamma(n+1)\Gamma(\alpha+1)}{}_1F_1(-n;\alpha+1;x).
\end{align}
Use \eqref{temme1} with $\alpha=-n-\frac{1}{2}$ in \eqref{I11} to arrive at
\begin{align}\label{I12}
I_1=\frac{e^{-x}}{\sqrt{x}}\sum_{n=0}^{\infty} \frac{(w^2/4)^n}{(1/2)_n }L_n^{-n-\frac{1}{2}}(x).
\end{align}
Now \cite[vol. 2, p. 706, \textbf{Formula (11)}]{prudnikov}
\begin{align}\label{prudnikov}
\sum_{k=0}^{\infty}\frac{t^k}{(\beta)_k}L_k^{\alpha-k}(x)=\Phi_3(-\alpha;\beta;-t,-tx).
\end{align}
Use \eqref{prudnikov} with $\beta=\frac{1}{2},\alpha=-\frac{1}{2},t=\frac{w^2}{4}$ so that from \eqref{I12},
\begin{align}\label{I1Final}
I_1=\frac{e^{-x}}{\sqrt{x}}\Phi_3\left(\frac{1}{2};\frac{1}{2};-\frac{w^2}{4},-\frac{w^2}{4}x\right)
\end{align}
Now we evaluate $I_2$. Applying \eqref{duplication} twice in \eqref{I2} then in next step interchanging order of series and integration by using the absolute convergence, we see that
{\allowdisplaybreaks\begin{align*}
I_2&=\sum_{r=1}^{\infty}\frac{\left(\frac{w^4}{16}\right)^r}{r!\left(\frac{1}{2}\right)_r\left(\frac{1}{2}\right)_{2r}}\frac{1}{2\pi i}\int_{(c)}\frac{\sqrt{\pi}}{2^{2r+s-\frac{3}{2}}}\frac{2^{s-\frac{3}{2}}}{\sqrt{\pi}}\frac{\Gamma\left(s-\frac{1}{2}\right)\Gamma\left(s-\frac{1}{2}+2r\right)}{\Gamma\left(s-\frac{1}{2}\right)} \nonumber\\
&\qquad\qquad \times {}_1F_1\left(s+2r,\frac{1}{2}+2r;-\frac{w^2}{4}\right)x^{-s} ds \nonumber\\
&=\sum_{r=1}^{\infty}\frac{\left(\frac{w^4}{16}\right)^r}{2^{2r}r!\left(\frac{1}{2}\right)_r\left(\frac{1}{2}\right)_{2r}}\frac{1}{2\pi i}\int_{(c)} \Gamma\left(s-\frac{1}{2}+2r\right)\sum_{m=0}^{\infty}\frac{\Gamma(s+2r+m)}{\Gamma(s+2r)\left(\frac{1}{2}+2r\right)_m}\frac{\left(\frac{-w^2}{4}\right)^m}{m!}x^{-s} ds \nonumber\\
&=\sum_{r=1}^{\infty}\frac{\left(\frac{w^4}{16}\right)^r}{2^{2r}r!\left(\frac{1}{2}\right)_r\left(\frac{1}{2}\right)_{2r}} \sum_{m=0}^{\infty} \frac{\left(\frac{-w^2}{4}\right)^m}{m!\left(\frac{1}{2}+2r\right)_m}\frac{1}{2\pi i}\int_{(c)} \Gamma\left(s-\frac{1}{2}+2r\right) \frac{\Gamma(s+2r+m)}{\Gamma(s+2r)} x^{-s} ds \nonumber\\
&=\sum_{r=1}^{\infty}\frac{\left(\frac{w^4}{16}\right)^r}{2^{2r}r!\left(\frac{1}{2}\right)_r\left(\frac{1}{2}\right)_{2r}} \sum_{m=0}^{\infty} \frac{\left(\frac{-w^2}{4}\right)^m}{m!\left(\frac{1}{2}+2r\right)_m}\frac{\sqrt{\pi}x^{2r}\sec(m\pi)}{\sqrt{x}\Gamma\left(\frac{1}{2}-m\right)}{}_1F_1\left(\frac{1}{2};\frac{1}{2}-m;-x\right) \nonumber\\
&=\sum_{r=1}^{\infty}\frac{\left(\frac{w^4}{16}\right)^r\sqrt{\pi}x^{2r-\frac{1}{2}}}{2^{2r}r!\left(\frac{1}{2}\right)_r\left(\frac{1}{2}\right)_{2r}} \sum_{m=0}^{\infty} \frac{\left(\frac{w^2}{4}\right)^m}{m!\left(\frac{1}{2}+2r\right)_m}\frac{e^{-x}{}_1F_1\left(-m;\frac{1}{2}-m;x\right)}{\Gamma\left(\frac{1}{2}-m\right)}, \nonumber\\
\end{align*}}%
where in the third step we again employed Lemma \ref{InverseMellinTransform}. Now use \eqref{temme1} with $\alpha=-\frac{1}{2}-m$ in above equation to see that
\begin{align}\label{I2Final}
I_2&=\frac{e^{-x}}{\sqrt{x}}\sum_{r=1}^{\infty}\frac{\left(\frac{w^4}{16}\right)^r\sqrt{\pi}x^{2r}}{2^{2r}r!\left(\frac{1}{2}\right)_r\left(\frac{1}{2}\right)_{2r}} \sum_{m=0}^{\infty} \frac{\left(\frac{w^2}{4}\right)^m}{m!\left(\frac{1}{2}+2r\right)_m} \frac{\Gamma(m+1)\Gamma\left(\frac{1}{2}-m\right)}{\Gamma\left(\frac{1}{2}-m\right)\Gamma\left(\frac{1}{2}\right)} L_m^{-m-\frac{1}{2}}(x) \nonumber\\
&=\frac{e^{-x}}{\sqrt{x}}\sum_{r=1}^{\infty}\frac{\left(\frac{w^4}{16}\right)^rx^{2r}}{2^{2r}r!\left(\frac{1}{2}\right)_r\left(\frac{1}{2}\right)_{2r}} \sum_{m=0}^{\infty} \frac{\left(\frac{w^2}{4}\right)^m}{\left(\frac{1}{2}+2r\right)_m}L_m^{-m-\frac{1}{2}}(x) \nonumber\\
&=\frac{e^{-x}}{\sqrt{x}}\sum_{r=1}^{\infty}\frac{\left(\frac{w^4}{16}\right)^rx^{2r}}{2^{2r}r!\left(\frac{1}{2}\right)_r\left(\frac{1}{2}\right)_{2r}} \Phi_3\left(\frac{1}{2};\frac{1}{2}+2r;-\frac{w^2}{4},-\frac{w^2}{4}x\right),
\end{align}%
where we again used \eqref{prudnikov} with $\alpha=-1/2,\ \beta=1/2+2r,\ t=-w^2/4$ in the last line.\\
Finally, from \eqref{I}, \eqref{I1Final} and \eqref{I2Final}, we have
\begin{align*}
K_{\frac{1}{2},w}(x)&=\sqrt{\frac{\pi}{2x}}e^{-x}\sum_{r=0}^{\infty}\frac{\left(\frac{w^4x^{2}}{64}\right)^r}{r!\left(\frac{1}{2}\right)_r\left(\frac{1}{2}\right)_{2r}} \Phi_3\left(\frac{1}{2};\frac{1}{2}+2r;-\frac{w^2}{4},-\frac{w^2}{4}x\right).
\end{align*}
\end{proof}

\section{Generalization of the transformation formula for the logarithm of the Dedekind eta function}
\begin{proof}
We let $z\rightarrow 1$ in Theorem \ref{DixRamGuin}. Let us concentrate first on the right hand side of \eqref{genRamGuin}.
	\begin{align*}
	&\lim\limits_{z\rightarrow 1} \frac{1}{4}\Gamma\left(\frac{z}{2}\right)\zeta(z)\left\{b^{\frac{1-z}{2}}{}_1F_1\left(\frac{1-z}{2};\frac{1}{2};\frac{w^2}{4}\right)-a^{\frac{1-z}{2}}{}_1F_1\left(\frac{1-z}{2};\frac{1}{2};-\frac{w^2}{4}\right)\right\} \nonumber\\
	&=\lim\limits_{z\rightarrow 1}\frac{1}{4}\Gamma\left(\frac{z}{2}\right)(z-1)\zeta(z)\left(\frac{b^{\frac{1-z}{2}}{}_1F_1\left(\frac{1-z}{2};\frac{1}{2};\frac{w^2}{4}\right)-a^{\frac{1-z}{2}}{}_1F_1\left(\frac{1-z}{2};\frac{1}{2};-\frac{w^2}{4}\right)}{z-1}\right) \nonumber
	\end{align*}
	Observe that the expression in the parenthesis of the above equation is of the form $\frac{0}{0}$ as $z$ tends 1. Therefore, by L'Hopital's rule and the fact that $\lim\limits_{z\rightarrow 1}(z-1)\zeta(z)=1$, we see that
{\allowdisplaybreaks	\begin{align}\label{lim1}
	&\lim\limits_{z\rightarrow 1} \frac{1}{4}\Gamma\left(\frac{z}{2}\right)\zeta(z)\left\{b^{\frac{1-z}{2}}{}_1F_1\left(\frac{1-z}{2};\frac{1}{2};\frac{w^2}{4}\right)-a^{\frac{1-z}{2}}{}_1F_1\left(\frac{1-z}{2};\frac{1}{2};-\frac{w^2}{4}\right)\right\} \nonumber \\
	&=\frac{\Gamma\left(\frac{1}{2}\right)}{4}\lim\limits_{z\rightarrow 1}\bigg[ b^{\frac{1-z}{2}}\frac{\log b}{-2}{}_1F_1\left(\frac{1-z}{2};\frac{1}{2};\frac{w^2}{4}\right)+b^{\frac{1-z}{2}}\frac{d}{dz}{}_1F_1\left(\frac{1-z}{2};\frac{1}{2};\frac{w^2}{4}\right) \nonumber\\
	&\qquad-a^{\frac{1-z}{2}}\frac{d}{dz}{}_1F_1\left(\frac{1-z}{2};\frac{1}{2};-\frac{w^2}{4}\right) - a^{\frac{1-z}{2}} \frac{\log a}{-2} {}_1F_1\left(\frac{1-z}{2};\frac{1}{2};-\frac{w^2}{4}\right)\bigg]  \nonumber\\
	&=\frac{\sqrt{\pi}}{4}\left[-\frac{1}{2}\log b-\frac{w^2}{4}{}_2F_2\left(1,1;\frac{3}{2},2;\frac{w^2}{4}\right)+\frac{1}{2}\log a-\frac{w^2}{4}{}_2F_2\left(1,1;\frac{3}{2},2;-\frac{w^2}{4}\right)\right] \nonumber\\
	&=\frac{\sqrt{\pi}}{8}\left[\log \frac{a}{b} -\frac{w^2}{2}\left({}_2F_2\left(1,1;\frac{3}{2},2;\frac{w^2}{4}\right)+{}_2F_2\left(1,1;\frac{3}{2},2;-\frac{w^2}{4}\right)\right)\right],
	\end{align}}%
	where in the penultimate line we employed \cite[Equation (5.5)]{DixitAnalogue}
	$$\lim\limits_{s\rightarrow 1} \frac{d}{ds}{}_1F_1\left(\frac{1-s}{2};\frac{1}{2};\frac{z^2}{4}\right)=-\frac{1}{4}z^2{}_2F_2\left(1,1;\frac{3}{2},2;\frac{1}{4}z^2\right).$$
	Also,
	\begin{align}\label{lim21}
	&\lim\limits_{z\rightarrow 1} \frac{1}{4}\Gamma\left(-\frac{z}{2}\right)\zeta(-z)\left\{b^{\frac{1+z}{2}}{}_1F_1\left(\frac{1+z}{2};\frac{1}{2};\frac{w^2}{4}\right)-a^{\frac{1+z}{2}}{}_1F_1\left(\frac{1+z}{2};\frac{1}{2};-\frac{w^2}{4}\right)\right\} \nonumber \\
	&=\frac{1}{4}\Gamma\left(-\frac{1}{2}\right)\zeta(-1)\left\{b\ {}_1F_1\left(1,\frac{1}{2};\frac{w^2}{4}\right)-a\  {}_1F_1\left(1;\frac{1}{2};-\frac{w^2}{4}\right)\right\}. 
	\end{align}
From \cite[p. 164, Formula 7.11.4]{nist}, we have
\begin{align*}
\textrm{erf}(z)=\frac{2z}{\sqrt{\pi}}{}_1F_1\left(\frac{1}{2};\frac{3}{2};-z^2\right)
\end{align*}
Use it with $z=\frac{w^2}{4}$ and apply $\eqref{kummer}$ to see
{\allowdisplaybreaks\begin{align*}
{}_1F_1\left(1;\frac{3}{2};\frac{w^2}{4}\right)&=\frac{\sqrt{\pi}}{w}e^{\frac{w^2}{4}}\textrm{erf}\left(\frac{w}{2}\right),
\end{align*}
so that
\begin{align}\label{1f1erf}
{}_1F_1\left(1;\frac{1}{2};\frac{w^2}{4}\right)&=1+\frac{\sqrt{\pi}w}{2}e^{\frac{w^2}{4}}\textrm{erf}\left(\frac{w}{2}\right).
\end{align}}%
From \eqref{erf} and \eqref{erfi} it is clear that $\mathrm{erf}(iz)=i\ \mathrm{erfi}(z)$. Therefore from \eqref{1f1erf}, we have
\begin{align}\label{1f1erfi}
{}_1F_1\left(1;\frac{1}{2};-\frac{w^2}{4}\right)&=1-\frac{\sqrt{\pi}w}{2}e^{-\frac{w^2}{4}}\textrm{erfi}\left(\frac{w}{2}\right).
\end{align}
Since $\Gamma\left(-\frac{1}{2}\right)=-2\sqrt{\pi},\ \zeta(-1)=-\frac{1}{12}$, and  from \eqref{lim21}, \eqref{1f1erf} and \eqref{1f1erfi}, we arrive at
\begin{align}\label{lim2}
&\lim\limits_{z\rightarrow 1} \frac{1}{4}\Gamma\left(-\frac{z}{2}\right)\zeta(-z)\left\{b^{\frac{1+z}{2}}{}_1F_1\left(\frac{1+z}{2};\frac{1}{2};\frac{w^2}{4}\right)-a^{\frac{1+z}{2}}{}_1F_1\left(\frac{1+z}{2};\frac{1}{2};-\frac{w^2}{4}\right)\right\} \nonumber \\
	&=\frac{\sqrt{\pi}}{24}\left[b\left(1+\frac{\sqrt{\pi}w}{2}e^{\frac{w^2}{4}}\textrm{erf}\left(\frac{w}{2}\right)\right)-a\left(1-\frac{\sqrt{\pi}w}{2}e^{-\frac{w^2}{4}}\textrm{erfi}\left(\frac{w}{2}\right)\right)\right],
\end{align}	 
	Now let $z\rightarrow 1$ on the left-hand side of \eqref{genRamGuin}, from \eqref{asymptoticlarge} we can interchange the order of limit and summation, thereby obtaining,
{\allowdisplaybreaks	\begin{align}\label{final1}
	&\lim\limits_{z\rightarrow 1}\left(\sqrt{a}\sum_{n=1}^\infty \sigma_{-z}(n)n^{z/2}e^{-\frac{w^2}{4}}K_{\frac{z}{2},iw}(2na)-\sqrt{b}\sum_{n=1}^\infty \sigma_{-z}(n)n^{z/2}e^{\frac{w^2}{4}}K_{\frac{z}{2},w}(2nb)\right)\nonumber\\
	&=e^{-\frac{w^2}{4}}\sum_{n=1}^\infty \sum_{r=0}^{\infty}\frac{\sigma_{-1}(n)e^{-2na}}{r!\left(\frac{1}{2}\right)_r\left(\frac{1}{2}\right)_{2r}}\Phi_3\left(\frac{1}{2};\frac{1}{2}+2r;\frac{w^2}{4},\frac{w^2n a}{2}\right)\left(-\frac{wna}{4}\right)^{2r}\nonumber\\
	&\qquad- e^{\frac{w^2}{4}}\sum_{n=1}^\infty \sum_{r=0}^{\infty}\frac{\sigma_{-1}(n)e^{-2nb}}{r!\left(\frac{1}{2}\right)_r\left(\frac{1}{2}\right)_{2r}}\Phi_3\left(\frac{1}{2};\frac{1}{2}+2r;-\frac{w^2}{4},-\frac{w^2n b}{2}\right)\left(\frac{wnb}{4}\right)^{2r}.
\end{align}}%
where we used Theorem \ref{K_(half,w)}.\\
	From \eqref{lim1}, \eqref{lim2} and  \eqref{final1}, we arrive at Theorem \ref{genOfRamanujanFormula}.
	\end{proof}



\section{Concluding remarks}


While reduction formulas for $\Phi_3(b;c;x,y)$ are available in the literature for some special values of $b$ and $c$, for example, \cite[Equation (5.2)]{brychkov} 
\begin{align*}
\Phi_3\left(1;\frac{3}{2};w,z\right)=\frac{\sqrt{\pi} e^{w+\frac{z}{w}}}{4\sqrt{w}}\left[\textrm{erf}\left(\frac{w-\sqrt{z}}{\sqrt{w}}\right)+\textrm{erf}\left(\frac{w+\sqrt{z}}{\sqrt{w}}\right)\right],
\end{align*}
as of yet, no reduction formulas are known for the Humbert function 
\begin{align}\label{reduction}
\Phi_3\left(\frac{1}{2};\frac{1}{2}+2r;-\frac{w^2}{4},-\frac{w^2}{4}x\right), \ r\geq 0\ r\in\mathbb{Z},
 \end{align} 
 that we have encountered in our work. So it is of interest to find the reduction formulas for \eqref{reduction}, for, it will lead to a simplification of Theorem \ref{K_(half,w)}.
 
 The error functions and the ${}_2F_2\left(1,1;\frac{3}{2},2;-w^2\right)$ on the right-hand side of Theorem \ref{genOfRamanujanFormula} remind us of the cumulative distribution function (CDF) corresponding to the Voigt profile. This observation is now explained. Note that the Voigt profile $V(x;\sigma,\beta)$ is the convolution of a Gaussian function $f_\sigma(x)$ with a Lorentzian function $L_\beta(x)$, i.e., 
  \begin{align*}
  V(x;\sigma,\beta):=\int_{-\infty}^{\infty}f_\sigma(t)L_\beta(x-t)dt,
  \end{align*}
  where \cite[p. 95, Equation (4)]{Grim-Stir}
  \begin{align*}
  f_\sigma(x):=\frac{e^{-\frac{x^2}{2\sigma^2}}}{\sigma \sqrt{2\pi}},\ -\infty<x<\infty,\ \sigma>0,
  \end{align*}
  and \cite[Equation (4)]{Hesse}
  \begin{align*}
  L_\beta(x):=\frac{\beta}{\pi(x^2+\beta^2)},\ \beta>0,
  \end{align*}
  where $\sigma$ and $\beta$ are the Gaussian- and Lorentzian-component widths respectively.
  Then
  \begin{align}\label{voigtFadd}
  V(x;\sigma,\beta)&=\frac{\beta}{\sigma\pi\sqrt{2\pi}}\int_{-\infty}^{\infty}\frac{e^{-\frac{t^2}{2\sigma^2}}}{(x-t)^2+\beta^2}dt\nonumber\\
 &=\frac{\mathrm{Re}\left(\xi(w)\right)}{\sigma\sqrt{2\pi}}.
  \end{align}
  where we used \cite[p. 2, Equation (2)]{Ida} and $w=\frac{x+i\beta}{\sqrt{2}\sigma}$ and $\xi(y)$ is the Faddeeva function defined by \cite[p. 297, Formula 7.1.3]{Abra-Stegun}
  \begin{align}\label{faddeeva}
  \xi(y)=e^{-y^2}\left(1-\textrm{erf}(-iy)\right).
  \end{align}
 The CDF corresponds to Voigt profile is given by
 {\allowdisplaybreaks\begin{align}\label{distrbn}
 F(x_0;\sigma,\beta)&=\int_{-\infty}^{x_0}V(x;\sigma,\beta)dx=\int_{-\infty}^{x_0}\frac{\mathrm{Re}(\xi(w))}{\sigma\sqrt{2\pi}}dx\nonumber\\
 &=\mathrm{Re}\left(\frac{1}{\sqrt{\pi}}\int_{\frac{-\infty+i\beta}{\sigma\sqrt{2}}}^{\frac{x_0+i\beta}{\sigma\sqrt{2}}}\xi(w)dw\right)\nonumber\\
 &=\mathrm{Re}\left(\frac{1}{\sqrt{\pi}}\int_{\frac{-\infty+i\beta}{\sigma\sqrt{2}}}^{\frac{x_0+i\beta}{\sigma\sqrt{2}}}e^{-w^2}(1-\mathrm{erf}(-iw))dw\right),
 \end{align}}%
 where in the last step we used \eqref{faddeeva}. From \cite[p. 35, Formula 1.5.3.4]{prudnikov}
 \begin{align}\label{int2}
 \frac{1}{\sqrt{\pi}}\int e^{-w^2}\mathrm{erf}(-iw)dw=-\frac{iw^2}{\pi}{}_2F_2\left(1,1;\frac{3}{2},2;-w^2\right).
 \end{align}
 It is easy to see that
 \begin{align}\label{int3}
 \lim\limits_{w\rightarrow-\frac{\infty+i\beta}{\sigma\sqrt{2}}}\mathrm{erf}(w)=-\frac{1}{2},
 \end{align}
 and
 \begin{align}\label{int4}
 \lim\limits_{w\rightarrow-\frac{\infty+i\beta}{\sigma\sqrt{2}}}\frac{iw^2}{\pi}{}_2F_2\left(1,1;\frac{3}{2},2;-w^2\right)=i\infty.
 \end{align}
 Therefore from \eqref{distrbn}, \eqref{int2}, \eqref{int3} and \eqref{int4}, we have
 \begin{align}\label{cdf}
 F(x;\sigma,\beta)=\mathrm{Re}\left[\frac{1}{2}+\frac{\mathrm{erf}(w)}{2}+\frac{iw^2}{\pi}{}_2F_2\left(1,1;\frac{3}{2},2;-w^2\right)\right].
 \end{align}
 It would be worthwhile to see if Theorem \ref{genOfRamanujanFormula} has applications in the study of Voigt profile. For further information on Voigt profile, we refer the reader to \cite{Hesse}, \cite{huang}, \cite{Ida} and \cite{thom} and the references therein.


\section*{Acknowledgements}
The author is very thankful to Professor Atul Dixit for suggesting this problem and for his various important suggestions throughout this work. He also thanks to Shivam Dhama for fruitful discussion on cummulative distribution functions.

\end{document}